\documentclass{article}
\usepackage{amssymb}

\title{\bf The Solubility Graph Associated With a Finite Group}
\author{{\sc B. Akbari,   Mark L. Lewis,  J. Mirzajani  and} \\[0.2cm]  {\sc A. R. Moghaddamfar}
\\[0.3cm] {\small Dedicated to Professor M. R. Darafsheh on the occasion of his 70th birthday}}

\newenvironment{proof}{\noindent {\em {Proof}}.}{$\square$
\medskip}
\newtheorem{theorem}{Theorem}[section]

\newtheorem{corollary}[theorem]{Corollary}

\newtheorem{remark}[theorem]{Remark}
\newtheorem{lm}[theorem]{Lemma}

\begin{document}
\maketitle
\begin{abstract}

\noindent Let $G$ be a finite group. The solubility graph associated with the finite group $G$, denoted by $\Gamma_{\cal S}(G)$, is a simple graph whose vertices are the non-trivial elements of $G$, and there is an edge between two distinct elements $x$ and $y$ if and only if $\langle x, y\rangle$ is a soluble subgroup of $G$. In this paper, we examine some properties of solubility graphs.
\end{abstract}

\renewcommand{\baselinestretch}{1.1}
\def\thefootnote{ \ }
\footnotetext{{\em $2010$ Mathematics Subject Classification}:
20D10, 20D05, 20D20, 05C25.}

%%%%%%%%%%%%%%%%%%%%%%%%%%%%%%%%%%%%%%%%%%%%%%%%%%%%%%%%%

\section{Introduction and Motivation}

All groups considered in this paper are assumed to be finite.  We will follow a graph theory approach, here.  Given a group $G$, we define {\em solubility graph} of $G$ to be the graph whose vertex set is $G$ and there is an edge between $x$ and $y$ when $\langle x, y\rangle$ is soluble.  We denote this graph by $\Gamma_{\cal S} (G)$.  For a background in graph theory, we suggest the reader consult \cite{bondy}.

One of the more interesting results regarding solvable groups is due to J. Thompson \cite{Thompson} and states that $G$ is soluble if and only if for every $x, y\in  G$ the subgroup $\langle x, y\rangle$ is soluble.  This result translates nicely to the graph as $G$ is soluble if and only if $\Gamma_{\cal S} (G)$ is complete.

We use $R(G)$ to denote the {\it soluble radical} of $G$, which is the largest soluble normal subgroup of $G$.  R. Guralnick, K. Kunyavski$\rm \check i$, E. Plotkin and A. Shalev proved in \cite{GKPS} that if $x$ is an element of the group $G$, then  $x\in R(G)$ if and only if the subgroup $\langle x, y\rangle$ is soluble for all $y\in G$.  In terms of the graph, this theorem translates to if $x$ is an element of the group $G$, then  $x\in R(G)$ if and only if $x$ is a universal vertex of  $\Gamma_{\cal S} (G)$ where a {\it universal vertex} is a vertex that is adjacent to every other vertex in the graph.  To understand the connectivity of the graph, it is useful to omit the universal vertices.

With this in mind, we define $\Delta_{\cal S}(G)$ to be the subgraph of  $\Gamma_{\cal S} (G)$ that is induced by the set $G \setminus R(G)$.  In Problem 3.1 of \cite{BNN}, they ask if $\Delta_{\cal S} (G)$ is connected for all groups $G$.  In our first theorem, we prove that this is true.

\begin{theorem}\label{main one}
If $G$ is a group, then $\Delta_{\cal S} (G)$ is connected.
\end{theorem}

The authors of \cite{BNN} ask in Problem 3.2 if there is a bound on the diameter of $\Delta_{\cal S} (G)$ when it is connected.  As part of our argument in Theorem \ref{main one}, we will show that the diameter of $\Delta_{\cal S} (G)$ is at most $11$.  However, we do not have groups that come anywhere near this bound, and we believe that the correct bound is probably much smaller.

For most of this paper, we focus on the sets of neighbors of elements in $\Gamma_{\cal S} (G)$.  With this in mind, we define for an element $x \in G$, the set $\mathsf{Sol}_G (x) = \{ g \in G \mid \langle x, g \rangle \ \mbox{is soluble} \}.$  We call this set the {\em solubilizer} of $x$ in $G$.  In general, this set will not be a subgroup of $G$.  Note that $\mathsf {Sol}_G (x) = G$ if and only if $x \in R(G)$.  We now consider how restrictions on the structure of this set influence the structure of $G$.  

\begin{theorem}\label{introtwo}
Let $G$ be a group.  If there exists an element $x \in G$ so that the elements of $\mathsf{Sol}_G (x)$ commute pairwise, then $G$ is abelian.
\end{theorem}

Hence, if $G$ has an element $x$ so that $\mathsf{Sol}_G (x)$ is an abelian subgroup of $G$, then $G$ is an abelian group.  As we mentioned above, $\mathsf{Sol}_G (x)$ need not be a subgroup.  It is natural to weaken the previous statement to ask what can be said if $\mathsf{Sol}_G (x)$ is a subgroup of $G$ for some element $x \in G$.  We note that this does not imply that $G$ must be soluble.  Consider $A_5$ and observe that if $x$ is an element of $A_5$ whose order is $3$ or $5$, then $\mathsf{Sol}_G (x)$ will be a subgroup of $G$.  However, when we assume that {\em all} of the solubizers are subgroups, we do indeed see that the group must be soluble.

\begin{theorem}\label{introthree}
Let $G$ be a group.  Then $G$ is soluble if and only if $\mathsf{Sol}_G (x)$ is a subgroup of $G$ for all $x \in G$.
\end{theorem}

It now makes sense to ask what other conditions we can put on the solubilizers that will force the group to be soluble.  We conclude by presenting two conditions on the conjugacy classes that will imply the group is soluble.

\begin{theorem}\label{introfour}
Let $G$ be a group.  Then the following are equivalent:
\begin{enumerate}
\item[{\rm 1.}] $G$ is soluble.
\item[{\rm 2.}]  For each conjugacy class ${\cal C}$ of $G$, the induced subgraph $\Gamma_{\cal S}({\cal C})$ is a clique.
\item[{\rm 3.}]  $\mathsf{Sol}_G(x) \cap {\cal C} \neq \emptyset$ for every element $x\in G$ and every conjugacy class ${\cal C}$ of $G$.
\end{enumerate}
\end{theorem}

\section{Basic Lemmas}

In this section, we state some elementary properties of solubilizers that will be useful later. 

\begin{lm}\label{union}
If $x$ is an element of the group $G$, then we have 
$$\langle x\rangle \subseteq \langle x, Z(G)\rangle \subseteq C_G(x)\subseteq N_G(\langle x\rangle)\subseteq N_G(\langle x\rangle) \cup R(G) \subseteq  \mathsf{Sol}_G(x)=\bigcup_{H} H,$$ 
where the union ranges over all soluble subgroups $H$ of $G$ containing $x$.
\end{lm}

We omit the straightforward proof.

\begin{corollary}\label{orderdivisor}
If $x$ is an element of the group $G$, then $|\mathsf{Sol}_G(x)|$ is divisible by $|x|$.
\end{corollary}

\begin{proof} Use Lemma \ref{union} and note that 
\begin{equation}\label{e1}|\mathsf{Sol}_G(x)|=\left|\cup_{i} H_i\right|=\sum_{i} |H_i|-\sum_{i<j}{|H_i\cap H_j}|+\sum_{i<j<k} |H_i\cap H_j\cap
H_k|-\cdots,\end{equation} where the $H_i$'s are soluble subgroups of $G$ containing $x$. The result now follows from the fact that $|x|$ divides  the right-hand side of (\ref{e1}).
\end{proof}

\begin{remark} 
{\rm An alternate proof proceeds by observing that the action of $\langle x\rangle$ on $\mathsf{Sol}_G(x)$ by right multiplication is semiregular, and so  all orbits have the same size. This shows that $|\mathsf{Sol}_G(x)|$ is divisible by $|\langle x\rangle|=|x|$, as required. }
\end{remark}

\begin{lm}\label{radicaldivisor}{\rm (\cite[Lemma 2.8]{doron})}
Let $N$ be a soluble normal subgroup of a group $G$ and $x \in G$.  Then $|\mathsf{Sol}_G(x)|$ is divisible by $|N|$.  In particular,  $|\mathsf{Sol}_G(x)|$ is divisible by $|R(G)|$.
\end{lm}

\begin{proof} 
Note that  $N$ acts on $\mathsf{Sol}_G(x)$  by right multiplication.
We remark that if $y\in N$ and $g\in \mathsf{Sol}_G(x)$, then  $\langle gy, x\rangle\leqslant N  \langle g, x
\rangle$. As $N$ and $ \langle g, x\rangle$ both are soluble,
$N  \langle g, x\rangle$ is soluble. It follows that $\langle gy, x\rangle$ is also soluble and hence $gy\in \mathsf{Sol}_G(x)$.  It is now easy to check that this action is semiregular, and so  
$|N|$ divides   $|\mathsf{Sol}_G(x)|$.  
\end{proof}

Let $N$ be a soluble normal subgroup of $G$ and $x \in G\setminus N$. We remark that $N\leqslant R(G)\subseteq  \mathsf{Sol}_G(x)$. 
Put $$\frac{\mathsf{Sol}_G(x)}{N}:=\{yN\ |\ y\in  \mathsf{Sol}_G(x) \}=\{yN \ | \ \langle y, x\rangle \ \mbox{is soluble}\}.$$ 
 Note that $ \mathsf{Sol}_G(x)$ is not necessarily a subgroup of $G$.
We claim that $$\left|\frac{\mathsf{Sol}_G(x)}{N}\right|=\frac{|\mathsf{Sol}_G(x)|}{|N|}.$$ To prove this,
let $N$ act on $\mathsf{Sol}_G(x)$ by right multiplication. Note that if $n\in N$ and $y\in  \mathsf{Sol}_G(x)$, then 
$yn\in  \mathsf{Sol}_G(x)$. Indeed, we have 
$\langle yn, x\rangle\leqslant \langle y, x\rangle N$
which is a soluble subgroup of $G$, and thus $\langle yn, x\rangle$ is also soluble. Hence $yn\in  \mathsf{Sol}_G(x)$.
Therefore $N$ permutes $\mathsf{Sol}_G(x)$ and partitions this set into orbits $yN$ with $y\in  \mathsf{Sol}_G(x)$. Clearly, $N$ has exactly 
 $|\mathsf{Sol}_G(x)/N|$ orbits on $\mathsf{Sol}_G(x)$, and the claim follows.

The following lemma is taken from \cite{doron}.

\begin{lm}\label{quotient} 
Let $N$ be a soluble normal subgroup of a finite group $G$, and assume that $x \in G$.  Then we have  $\mathsf {Sol}_{G/N} (xN) = \mathsf {Sol}_G (x)/N$. 
\end{lm}

\begin{proof} 
We make use of the following fact:
$$\langle xN, yN\rangle=\frac{\langle x, y\rangle N}{N} \cong \frac{\langle x, y\rangle}{\langle x, y\rangle \cap N},$$
from which it follows that $\langle x, y\rangle$ is soluble if and only if $\langle xN, yN\rangle$ is soluble.
\end{proof}

\section{The Solubilizer}

We begin this section by proving the following theorem which is useful as well as interesting.

\begin{theorem}\label{MI}
Let $A$ be an abelian subgroup of a group $G$. If $A$ is maximal among soluble subgroups of $G$, then $A=G$.  In particular, if the elements of the solubilizer of some element in $G$ commute pairwise, then $G$ is abelian.
\end{theorem}

\begin{proof} 
We use induction on $|G|$.  We may assume that $A$ is a maximal subgroup of $G$, in fact, if there exists a subgroup $H$ of $G$ such that $A\leqslant H<G$, then the inductive hypothesis yields that $A=H$.

On the other hand, if there exists a normal subgroup $N$ of $G$ with $1<N\leqslant A$, then we may apply the inductive hypothesis to $G/N$ with respect to $A/N$ to conclude that  $A/N=G/N$, from which it follows that $A=G$. We may therefore assume that such a normal subgroup $N$ of $G$ does not exist. In particular, $A$ is not normal in $G$. Thus $N_G(A)=A$ by the maximality of $A$.  This implies that for all elements $g \in G\setminus A$, $A^g\neq A$. We claim that $A\cap A^g=1$ for all elements $g\in G\setminus A$. Consider the subgroup $\langle A, A^g\rangle$ of $G$. Since $A<\langle A, A^g\rangle$, so $\langle A, A^g\rangle=G$. It is obvious that $A\cap A^g\leqslant Z(G)$ and thus $A\cap A^g$ is normal in $A$ which forces that $A\cap A^g=1$, as claimed.

Let $X$ be the subset of $G$ consisting of those elements that are not conjugate in $G$ to any nonidentity element of $A$.  By Lemma 6.5 in \cite{Isaacs}, $|G|=|X||A|$. Also, by Frobenius' theorem,  the set $X$ is a normal subgroup of $G$, and by the definition of $X$, we see that $X \cap A=1$.  Therefore, $G=AX$, and thus, $A$ complements the normal subgroup $X$ in $G$. Hence, we can easily conclude that $G$ is a Frobenius group with kernel $X$, which is nilpotent by Thompson's theorem. Finally, $G$ is soluble, and hence $A=G$.

Now let $x\in G$ and the elements of $\mathsf {Sol}_G (x)$ commute pairwise. We claim that the solubilizer $\mathsf {Sol}_G (x)$ forms an abelian subgroup of $G$.  To see this, observe that $\langle x \rangle \subseteq \mathsf{Sol}_G(x)$, and so $\mathsf {Sol}_G (x)$ is non-empty. Furthermore, if $y, z\in \mathsf{Sol}_G(x)$, then $x$, $y$, and $z$ commute pairwise. Thus $\langle x, yz \rangle$  is an abelian subgroup of $G$, which is soluble, and so $yz \in \mathsf{Sol}_G(x)$. This shows that $\mathsf{Sol}_G(x)$ is an abelian subgroup of $G$, as claimed. We now show that  $\mathsf{Sol}_G(x)$  is maximal among all soluble subgroups of $G$. Suppose $H$ is a soluble subgroup of $G$ which contains $\mathsf{Sol}_G(x)$ properly. We may choose $y\in H\setminus \mathsf{Sol}_G(x)$ and consider the subgroup $\langle y, x\rangle$. But then $\langle y, x\rangle$ as a subgroup of $H$ is soluble, and this forces $y\in \mathsf{Sol}_G(x)$, which is a contradiction.  The result now follows by the first part of theorem. 
\end{proof}

As an immediate consequence of Corollary \ref{orderdivisor} and Theorem \ref{MI}, we have the following:

\begin{corollary}\label{xsol}
Let $G$ be an insoluble group, and assume that $x\in G$. Then $\langle x\rangle$ is properly contained in $\mathsf{Sol}_G(x)$. In particular,  
\begin{itemize}
\item [{\rm (1)}]  
There is a soluble subgroup $H$ of $G$ which contains $\langle x\rangle$ properly.
\item [{\rm (2)}] 
$|\mathsf{Sol}_G(x)|$ cannot be a prime number.
\end{itemize}
\end{corollary}

\begin{proof}   
Note that in light of Theorem \ref{MI} the solubilizer $\mathsf {Sol}_G (x)$ cannot admit the structure of an abelian group.  Hence, we must have $\langle x \rangle$ to be a proper subset of $\mathsf {Sol}_G (x)$.  Thus, there is an element $y \in \mathsf {Sol}_G (x) \setminus \langle x \rangle$.  The subgroup $\langle x, y \rangle$ will be soluble and properly contain $\langle x \rangle$.  By Corollary \ref{orderdivisor}, we know that $|x|$ divides $|\mathsf{Sol}_G(x)|$, and so, $|\mathsf{Sol}_G(x)|$ a prime number implies that $\mathsf{Sol}_G(x) = \langle x \rangle$, and we have seen that this is a contradiction.
\end{proof}

We next show that $|\mathsf{Sol}_G(x)|$  cannot be a square of a prime number under the additional hypothesis that $R (G) \ne 1$.  It is not clear that this additional hypothesis is really needed to obtain this conclusion.

\begin{corollary}\label{radical}
Let $G$ be an insoluble group such that $R(G) \ne 1$.  If $x \in G$, then $|\mathsf{Sol}_G(x)|$ cannot be a square of a prime number.
\end{corollary}

\begin{proof}  
Let $|\mathsf{Sol}_G(x)| = p^2$, where $p$ is a prime number.  First, since every group of order $p^2$ is abelian, the solubilizer $\mathsf{Sol}_G(x)$ does not admit a group structure by Theorem \ref{MI}.  It follows by Corollary \ref{orderdivisor} that the cyclic group $\langle x\rangle$ is of order $p$.  Similarly, by Lemma \ref{radicaldivisor}, we conclude that $|R(G)|=p$.  Clearly, $x$ does not lie in $R(G)$, and so $\langle x \rangle \cap R(G) = 1$. But then $R(G) \langle x \rangle$ is a group of order $p^2$, which forces $\mathsf {Sol}_G (x) = R(G) \langle x\rangle$, a contradiction. 
\end{proof} 

We now prove Theorem \ref{introtwo} from the Introduction which we restate here.  Recall that a group $G$ is {\it partitioned} if there exist proper, nontrivial subgroups $H_1, \dots, H_m$ so that $G = \cup_{i=1}^m H_i$ and $H_i \cap H_j = 1$ when $i \ne j$ for $1 \leqslant i,j \leqslant m$.  We will make use of Suzuki's classification of partitioned insoluble groups.

\begin{theorem}  
Let $G$ be a group. Then $G$ is soluble if and only if $\mathsf{Sol}_G(x)$ is a subgroup of $G$ for every element $x\in G$.
\end{theorem}

\begin{proof}  
If $G$ is soluble, then $\mathsf{Sol}_G(x) = G$ for every $x\in G$ and so $\mathsf{Sol}_G(x)$ is a subgroup of $G$ for every element $x\in G$. 

Conversely, assume that $\mathsf{Sol}_G(x)$ is a subgroup of $G$ for every element $x\in G$. We work by induction on $|G|$. Notice that if $R(G)=G$, then $G$ is soluble and there is nothing to prove. Assume then that $R(G)<G$.  If $H$ is a subgroup of $G$, then $\mathsf{Sol}_H(x) = \mathsf{Sol}_G(x)\cap H$ will be a subgroup of $H$ for every element $x\in  H$. If $H$ is proper in $G$, then we may apply the inductive hypothesis to see that $H$ is soluble. Thus, every proper subgroup of $G$ is soluble.  
 
Since $R(G)<G$, the quotient group $G/R(G)$ is an insoluble group.  We claim that if $x\in G \setminus  R(G)$, then $x$ lies in a unique maximal subgroup of $G$. Since $x\notin R(G)$, we know that $x\in  \mathsf{Sol}_G(x) < G$. Thus, $x$ is contained in a maximal subgroup of $G$, say $M$. Since $M<G$, we know that $M$ is soluble. If $y\in M$, then $\langle x, y\rangle\leqslant M$ and so, $\langle x, y\rangle$ is soluble. This implies that $y\in \mathsf{Sol}_G(x)$. Thus, $M \leqslant  \mathsf{Sol}_G(x)$. Since $\mathsf{Sol}_G(x)$ is a subgroup and $\mathsf{Sol}_G(x) < G$, we see that $\mathsf{Sol}_G(x)=M$. We note that as $M$ was arbitrary, this implies that $\mathsf{Sol}_G(x)$ is the unique maximal subgroup of $G$ containing $x$. It follows that $G/R(G)$ is partitioned by its maximal subgroups. 

Suzuki has classified the insoluble groups that are partitioned (\cite{Suzuki}). In particular, the possible groups are insoluble Frobenius groups, $L_2(q)$, where $q$ is a prime power greater than $3$, ${\rm PGL}_2(q)$ where $q$ is an odd prime power greater than $3$, and ${\rm Sz}(q)$, where $q = 2^{2f+1}$ for some integer $f \geqslant  1$. It is not difficult to see that these groups are not partitioned by their maximal subgroups. This yields a contradiction, and so, we must have $G$ is soluble.  
\end{proof}

To prove the next result, we make use of minimal simple groups.  A {\em minimal simple group} is a non-abelian simple group all of whose proper subgroups are soluble. Thompson \cite[Corollary 1]{Thompson}  has determined the minimal simple groups:  every minimal simple group is isomorphic to one of the following groups: $L_2(2^p)$, $L_2( 3^p)$, $L_2(p)$, ${\rm Sz}(2^p)$, where $p$ is an odd prime, $L_2(4)$ and $L_3(3)$.  The next result gives the first pair of equivalences in Theorem \ref{introfour}.

\begin{theorem} \label{clique}
A finite group $G$ is soluble if and only if for each conjugacy class ${\cal C}$ of $G$, the induced subgraph $\Gamma_{\cal S}({\cal C})$ is a clique.
\end{theorem}

\begin{proof}   
If $G$ is soluble, the conclusion is clear. 

Conversely,  suppose for every conjugacy class ${\cal C}$ of $G$, the induced subgraph $\Gamma_{\cal S}({\cal C})$ is a clique. The proof will be by contradiction and we let $G$ be a counterexample of minimal order.   Let $H$ be a proper subgroup of $G$ and $x\in H$.  If ${\cal C}_H$ and ${\cal C}_G$ denote the conjugacy classes of $H$ and $G$ containing $x$, respectively, then obviously ${\cal C}_H\subseteq {\cal C}_G$. It follows that $\Gamma_{\cal S}({\cal C}_H)$ is also a clique, and so $H$ satisfies the hypothesis of $G$.  Finally, the minimality of $G$ implies that every proper subgroup of $G$ is soluble. 

Next we claim that $G$ is a simple group. To prove this, let $N$ be a normal subgroup of $G$ and let $a, g \in G$.  It is easy to see that 
$$\langle aN, (aN)^{gN} \rangle  = \frac{\langle a, a^g \rangle N}{N} \cong \frac{\langle a, a^g \rangle}{\langle a, a^g \rangle \cap N}. $$ 
Notice that the assumption that $\Gamma_{\cal S} ({\cal C})$ is a clique for every conjugacy class ${\cal C}$ implies that $\langle a, a^g \rangle$ is soluble for all pairs of elements $a, g \in G$.  It follows that $\langle aN, (aN)^{gN} \rangle$ will be soluble.  We deduce that $G/N$ satisfies the hypotheses.  By the minimality of $G$, if $G > N > 1$, then $G/N$ is soluble.  As $N$ is also soluble, this implies that $G$ is soluble which is a contradiction.

Hence, we may assume that $G$ is simple and as noted, $G$ will be a minimal simple group.  It suffices to show for each minimal simple group that there is a conjugacy class so that $\Gamma_{\cal S} ({\cal C})$ is not a clique.  To do this, we claim that for each minimal simple group $S$, there is a prime $q$ so that $S$ has a cyclic Sylow subgroup $Q$ and all of the maximal subgroups containing $Q$ normalize $Q$.  

Assuming this claim is true, let $a$ be a generator for $Q$ and let $g$ be an element of $S$ that does not normalize $\langle a \rangle$.  Then it will follow that $\langle a, a^g \rangle$ contains at least two distinct Sylow $q$-subgroups, and so, it cannot be contained in any maximal subgroup of $S$.  This implies that $S = \langle a, a^g \rangle$.  Since $S$ is not soluble, this implies that $\Gamma_{\cal S} ({\cal C})$ is not a clique when ${\cal C}$ is the conjugacy class containing $a$.

When $S$ is $L_2 (2^p)$, we take $q$ to be a prime divisor of $2^p+1$; when $S$ is $L_2 (3^p)$, take $q$ to be an odd prime divisor of $3^p+1$; when $S$ is $L_2 (p)$, take $q = p$; when $S$ is $L_ 2(4)$, take $q= 5$; when $S$ is ${\rm Sz} (2^p)$, take $q$ to be a prime divisor of $2^{2p}+1$; and when $S$ is $L_3(3)$, take $q = 13$.  To see that these groups have the desired property, we use Dickson's classification of the subgroups of $L_2(q)$ \cite{Dickson} for those 
groups (see also Hauptsatz II.8.27, p. 213, of \cite{Huppert1}).  For the Suzuki groups, we can use Suzuki's original paper \cite{Suzuki-d} (see also Remark IX.3.12 in \cite{Huppert-Blackburn}).  The result for $L_ 3(3)$ can be read from the Atlas of Finite Groups \cite{atlas}. 
\end{proof}

\section{The Soluble Grueneberg-Kegel Graph}
We make a new definition and a few observations before going on to prove anything.  For a group $G$, we denote by $\pi(G)$ the set of prime divisors of $|G|$. We define a graph with $\pi(G)$ as its vertex set by linking $p, q\in \pi(G)$ if and only if there exists a soluble subgroup $H$ of $G$ whose order is divisible by $pq$. This is called the {\em soluble Grueneberg-Kegel graph} of $G$  and is denoted by  $\Gamma_{\rm s}(G)$.  For two primes $p, q\in \pi(G)$, we will write $p\approx q$ if $p$ and $q$ are adjacent in $\Gamma_{\rm s}(G)$. This graph was first introduced by Abe and Iiyori in \cite{Abe-Iiyori}.  In Theorems 1 and 2 in \cite{Abe-Iiyori}, they prove that if $G$ is a nonabelian simple group, then $\Gamma_{\rm s}(G)$ is connected, but not complete.

When $p\approx q$ in $\Gamma_{\rm s}(G)$, there exists, by the definition of $\Gamma_{\rm s}(G)$, a soluble subgroup $H$ of $G$ such that $|H|$ is divisible by $pq$.  Thus, we can find elements $x, y \in H$ with $|x| = p$, $|y|= q$, and $\langle x, y \rangle$ is soluble. Hence, $x \sim y$ in $\Gamma_{\cal S}(G)$.

We now prove the remaining equivalences for Theorem \ref{introfour}.

\begin{theorem} \label{conj int}
Let ${\cal C}_1, {\cal C}_2, \ldots, {\cal C}_k$ be the distinct conjugacy classes of a group $G$.  Then $G$ is soluble if and only if, $\mathsf{Sol}_G(x) \cap {\cal C}_i\neq \emptyset$ for every element $x\in G$ and for every integer $i$ satisfying $1 \leqslant i \leqslant k$.
\end{theorem}

\begin{proof}  
If $G$ is soluble, then for every element $x\in G$, we have $\mathsf{Sol}_G(x) = G$, and so $$\mathsf{Sol}_G(x)\cap {\cal C}_i=G\cap {\cal C}_i={\cal C}_i \neq \emptyset$$  for every integer $i$ satisfying $1\leqslant i\leqslant k$. 

Conversely, assume for every element $x \in G$ that $\mathsf{Sol}_G(x)\cap {\cal C}_i\neq \emptyset$ for  $1\leqslant i\leqslant k$. Let $G$ be a minimal counterexample to the claim (a minimal order insoluble group for which the condition holds).  We now consider the soluble Grueneberg-Kegel graph $\Gamma_{\rm s}(G)$ of $G$. We claim that $\Gamma_{\rm s}(G)$ is complete. To see this, let $p$ and $q$ be primes that divide $|G|$.  We want to show that $G$ has a soluble subgroup whose order is divisible by $pq$.  We can find elements $x, y \in G$ so that $|x|= p$ and $|y| = q$.  By our hypothesis, there exists an element $g \in G$ so that $H = \langle x, y^g \rangle$ is soluble.  It follows that $pq$ divides $|H|$ which proves the claim.  Thus, $\Gamma_{\rm s} (G)$ is a complete graph.

As we noted above, it now follows from \cite[Theorem 2]{Abe-Iiyori} that $G$ cannot be a nonabelian simple group. Hence, we may choose $N$ to be a proper minimal normal subgroup of $G$.  First, fix the coset $xN\in G/N$, and let ${\cal C}_{G/N}$ and ${\cal C}_G$ be the conjugacy classes of $G/N$ and $G$ containing $xN$ and $x$, respectively.  By hypothesis we have $\mathsf{Sol}_G(x) \cap {\cal C}_G\neq \emptyset$.  This means that for some $g$ in $G$, $\langle x, y^g\rangle$ is soluble. But then, we have 
$$\langle xN, (yN)^{gN}\rangle=\frac{\langle x, y^g\rangle N}{N}\cong \frac{\langle x, y^g\rangle}{\langle x, y^g\rangle \cap N},$$ 
which shows that $\langle xN, (yN)^{gN}\rangle$ is soluble. Thus, $G/N$ also satisfies the hypothesis of the theorem. By the minimality of $G$ it follows that $G/N$ is soluble. Thus, since $G$ is insoluble, $N$ is an insoluble minimal normal subgroup of $G$. 

Observe that $N\cong N_1\times N_2\times \cdots \times N_k$ where the $N_i$'s are isomorphic to a nonabelian simple group.  Write $P$ for the nonabelian simple group which is isomorphic to $N_i$ for $1\leqslant i\leqslant k$. In view of Theorem 2 of \cite{Abe-Iiyori}, $\Gamma_{\rm s}(P)$ is not complete, and so it contains two nonadjacent vertices, say $r$ and $s$.  Using the observation before this lemma, we have for all elements $x, y\in P$ of orders $r$ and $s$, respectively, that the subgroups $\langle x, y\rangle$ are insoluble.  Take $u = (u_1, u_2, \ldots, u_k)$ and $v = (v_1, v_2, \ldots, v_k)$ to be elements of $N$ with $|u_i|=r$ and $|v_i|=s$ for each $1\leqslant i\leqslant k$.  If the element $g \in G$ is arbitrary, then $\langle u, v^g\rangle$ is a subgroup of $N$ for which the projection to each direct factor $N_i$ of $N$ is a subgroup $\langle u_i, v_i^g\rangle$ with $|u_i| = p$, $|v_i^g| = q$, and hence is insoluble. In particular, $\langle u, v^g\rangle$ is insoluble. Since $g \in G$ was arbitrary, we conclude that $\mathsf{Sol}_G (u)$ has a trivial intersection with the conjugacy class of $G$ containing $v$, a contradiction. 
\end{proof}

We also will use the soluble Grueneberg-Kegel graph to prove Theorem \ref{main one} which we restate here.

\begin{theorem}\label{connected} 
For every group $G$, the solubility graph $\Delta_{\cal S}(G)$ is connected, and its diameter is at most $11$.
\end{theorem}

\begin{proof}  
First, we notice that $\Delta_{\cal S}(G)$ is connected if and only if $\Delta_{\cal S} (G/R(G))$ is connected.  Indeed, the point here is that 
$$\langle xR(G), yR(G) \rangle = \frac{\langle x, y \rangle R(G)}{R(G)} \cong \frac{\langle x, y 
\rangle}{\langle y, x \rangle \cap R(G)}.$$  
Since $\langle y, x \rangle \cap R(G)$ is always soluble, we conclude that  $\langle x, y \rangle$ is soluble if and only if $\langle xR(G), yR(G) \rangle$ is soluble, or equivalently, $x\sim y$ in $\Gamma_{\cal S}(G)$ if and only of $xR(G)\sim yR(G)$ in $\Gamma_{\cal S} (G/R(G))$. Thus, we may assume that $R(G) = 1$.

Let $I = {\rm Inv}(G)$ be the set of all involutions in $G$.  We know that any two involutions generate a dihedral group.  Since dihedral groups are always soluble, we see that $\Gamma_{\cal S}(I)$, the subgraph induced by $I$, is a complete graph.  Thus, it suffices to show that every element of $G \setminus I$ is connected to an involution.

Suppose $g\in G \setminus I$ is a nontrivial element.  We want a path from $g$ to an involution. First of all, there is a prime $p$ that divides $|g|$.  Hence, there is an integer $n$ so that $|g^n|= p$, and consequently  we have an edge between $g$ and $g^n$, that is $g\sim g^n$ in $\Delta_{\cal S}(G)$. On the other hand, using the fact that $\Gamma_{\rm s} (G)$ is connected (see \cite[Corollary 2]{Abe-Iiyori}), we can find a path, between $p$ and $2$ in $\Gamma_{\rm s} (G)$, say 
$$p=p_1 \ \approx \ p_2 \ \approx \ p_3 \ \approx  \  \cdots \ \approx  \ p_k = 2,$$  
Now, using the observation before Theorem \ref{conj int}, we can find the elements $y_i$ and $x_{i+1}$ in $G$ so that $|y_i|= p_i$ and $|x_{i+1}|= p_{i+1}$ and $y_i \sim x_{i+1}$ in $\Delta_{\cal S}(G)$.  Take $x_1 = g^n$.  Moreover, for $i = 1, \ldots, k-1$, observe that $|x_i|= |y_i| = p_i$.  Now, there exists an element $h_i \in G$ so that $x_i$ and $y_i^{h_i}$ lie in the same Sylow $p_i$-subgroup of $G$.  In particular, $\langle x_i, y_i^{h_i} \rangle$ is a $p_i$-group which implies it is soluble, and so $x_i\sim y_i^{h_i}$ in $\Delta_{\cal S}(G)$.  Furthermore, for each $i$, we observe that:
\begin{itemize}
\item[(1)] since $y_i\sim x_{i+1}$,  conjugating by $h_i\cdots h_1$ gives  
$y_i^{h_i\cdots h_1}\sim x_{i+1}^{h_i\cdots h_1}$, and
\item[(2)]  since $x_i\sim y_{i}^{h_i}$, conjugating by $h_{i-1}\cdots h_1$ gives  
$x_i^{h_{i-1}\cdots h_1}\sim y_i^{h_i(h_{i-1}\cdots h_1)}$.
\end{itemize}
Using the above observations, one can easily see that 
$$g\sim x_1 \sim  y_1^{h_1}  \sim \ x_2^{h_1}  \sim  y_2^{h_2 h_1}  \sim x_3^{h_2 h_1} \sim  y_3^{h_3 h_2 h_1} \ \sim \cdots  \sim \ x_k^{h_{k-1} \cdots h_{1}}$$ 
is a path in $\Delta_{\cal S} (G)$ from $g$ to the involution $x_k^{h_{k-1} \cdots h_{1}}$.  Hence, we have the desired path to an involution.   If $g_1$ and $g_2$ are any two elements in $G \setminus \{ 1 \}$, then we can find a path from $g_1$ to an involution $i_1$ and a path from $g_2$ to an involution $i_2$.  Since $i_1$ and $i_2$ are adjacent, we now obtain a path from $g_1$ and $g_2$.  This proves that $\Delta_{\cal S} (G)$ is connected.

In \cite[Theorem 2]{Hagie}, it is shown that the distance between $2$ and $p$ in $\Gamma_{\rm s}(G)$ is at most $3$ for any prime $p$ in $\pi (G)$.  Using this fact with the path above, one can show that every element of $G \setminus \{1\}$ has distance at most $5$ to an involution.  This shows that there is a path of length at most $11$ between any two elements of $G \setminus \{1\}$ in $\Delta_{\cal S} (G)$. 
\end{proof}

A group $G$ is said to be {\em soluble transitive} if for all $x, y, z\in G\setminus \{1\}$, the subgroups $\langle x, y \rangle$ and $\langle y, z \rangle$ soluble imply $\langle x, z\rangle$. In other words, if one defined the relation on $G$ that $x$ and $y$ are related if they generate a soluble, then the relation is transitive if and only if $G$ is soluble transitive.  In graph-theoretical terms, $G$ is a soluble transitive group precisely when every connected component of $\Gamma_{\cal S}(G \setminus \{ 1\})$ is a clique.  The following result, due to  Delizia, Moravec and Nicotera, characterizes the structure of ${\cal S}$-transitive groups (see \cite{DMN}): {\em A group is soluble transitive if and only if it is soluble}.  Notice that one can also obtain this conclusion by using Theorem \ref{main one} that $\Delta_{\cal S} (G)$ is connected with Thompson's theorem that $G$ is soluble if and only if $\Gamma_{\cal S} (G)$ is a complete graph.

\begin{center}
 {\sc Acknowledgments}
\end{center}
A part of this work was done during the first author had a visiting position at the
Faculty of Mathematics, K. N. Toosi University of Technology (June-September 2019). She would like to thank the hospitality of the Faculty of Mathematics of KNTUT.

\noindent {\sc B. Akbari}\\[0.2cm]
{\sc Department of Mathematics, Sahand University of Technology,
Tabriz, Iran.}\\[0.1cm]
{\em E-mail address}: {\tt  b.akbari@sut.ac.ir}\\[0.3cm]
\noindent {\sc Mark L. Lewis}\\[0.2cm]
{\sc Department of Mathematical Sciences, Kent State
University,}\\ {\sc  Kent, Ohio $44242$, United States of
America}\\[0.1cm]
{\em E-mail address}: {\tt  lewis@math.kent.edu}\\[0.3cm]
 {\sc J. Mirzajani and A. R. Moghaddamfar}\\[0.2cm]
{\sc Faculty of Mathematics, K. N. Toosi
University of Technology,
 P. O. Box $16765$--$3381$, Tehran, Iran,}\\[0.1cm]
{\em E-mail addresses}:  {\tt  jmirzajani@mail.kntu.ac.ir}, \  {\tt moghadam@kntu.ac.ir}\\[0.3cm]

\end{document}